\documentclass[11pt,reqno]{amsart}
\usepackage{amsmath} 
\usepackage[T1]{fontenc}
\usepackage[utf8]{inputenc}
\usepackage{lmodern}
\usepackage{microtype}
\usepackage{amsmath,amssymb,mathtools}
\usepackage{booktabs}
\usepackage{enumitem}
\usepackage{tabularx}
\usepackage{xcolor}
\usepackage{xurl}
\usepackage[numbers,sort&compress]{natbib}

\usepackage[pdfusetitle,colorlinks=true,
  pdfauthor={Hao Shen, Jiaqi Wang, and Lihong Zhi},
  linkcolor=blue!45!black,citecolor=blue!45!black,
  urlcolor=blue!45!black]{hyperref}

\newtheorem{theorem}{Theorem}[section]
\newtheorem{proposition}[theorem]{Proposition}
\newtheorem{lemma}[theorem]{Lemma}
\newtheorem{corollary}[theorem]{Corollary}
\theoremstyle{definition}
\newtheorem{definition}[theorem]{Definition}
\theoremstyle{remark}

\theoremstyle{plain}
  {Theorem~\ref*{thm:main2} (Counterexample to the diagonal route)}

\makeatletter
\renewcommand{\paragraph}{%
  \@startsection{paragraph}{4}{\z@}%
    {1.25ex \@plus .5ex \@minus .2ex}
    {-1em}
    {\normalfont\normalsize\bfseries}}
\makeatother
\newcommand{\Span}{\operatorname{span}}
\newcolumntype{Y}{>{\raggedright\arraybackslash}X}
\allowdisplaybreaks
\usepackage{xcolor}
\usepackage{listings}
\definecolor{keywordcolor}{rgb}{0.65,0.08,0.08}
\definecolor{symbolcolor}{rgb}{0.00,0.10,0.55}
\definecolor{sortcolor}{rgb}{0.05,0.42,0.05}
\definecolor{commentcolor}{rgb}{0.45,0.45,0.45}
\providecolor{keywordcolor}{rgb}{0.7,0.1,0.1}
\providecolor{tacticcolor}{rgb}{0.0,0.1,0.6}
\providecolor{commentcolor}{rgb}{0.4,0.4,0.4}
\providecolor{symbolcolor}{rgb}{0.0,0.1,0.6}
\providecolor{sortcolor}{rgb}{0.1,0.5,0.1}
\providecolor{attributecolor}{rgb}{0.7,0.1,0.1}
\providecolor{stringcolor}{rgb}{0.5,0.3,0.0}
\providecolor{errorcolor}{rgb}{1.0,0.0,0.0}
\lstnewenvironment{leancode}{\lstset{language=lean,abovecaptionskip=-\medskipamount}}{}

\newcommand{\ii}{\mathrm i}

\newcommand{\C}{\mathbb C}
\newcommand{\R}{\mathbb R}
\newcommand{\tr}{\operatorname{tr}}
\newcommand{\Rea}{\operatorname{Re}}
\newcommand{\Ima}{\operatorname{Im}}
\newcommand{\diag}{\operatorname{diag}}

\newcommand{\codim}{\operatorname{codim}}
\newcommand{\norm}[1]{\left\lVert#1\right\rVert}
\newcommand{\comm}[2]{[#1,#2]}
\newtheorem*{main2restatement}
  {Theorem~\ref{thm:main2} (restated)}

\numberwithin{equation}{section}
\allowdisplaybreaks[1]
\title[A Dimension-Independent Commutator Bound]{A Dimension-Independent Commutator Bound}
\author{Hao Shen}
\address{State Key Laboratory of Mathematical Sciences,
Academy of Mathematics and Systems Science, Chinese Academy of Sciences,
Beijing 100190, China; University of Chinese Academy of Sciences,
Beijing 100049, China}
\email{shenhao24@amss.ac.cn}

\author{Jiaqi Wang}
\address{State Key Laboratory of Mathematical Sciences,
Academy of Mathematics and Systems Science, Chinese Academy of Sciences,
Beijing 100190, China; University of Chinese Academy of Sciences,
Beijing 100049, China}
\email{jiaqiwang@amss.ac.cn}

\author{Lihong Zhi}
\address{State Key Laboratory of Mathematical Sciences,
Academy of Mathematics and Systems Science, Chinese Academy of Sciences,
Beijing 100190, China; University of Chinese Academy of Sciences,
Beijing 100049, China}
\email{lzhi@mmrc.iss.ac.cn}

\begin{document}

\maketitle
\begin{abstract}
We prove that every trace-zero matrix $A\in M_n(\C)$ admits
a representation $A=BC-CB$ with $B,C\in M_n(\C)$ and
$\norm{B}\norm{C}\le K\norm{A}$, where $K$ is an absolute
constant independent of $n$, and $\norm{\cdot}$ denotes
the operator norm. For a fixed $t>0$, the proof splits according to
whether $\norm{\Rea(e^{\ii\theta}A)}_1\ge tn\norm{A}$ holds for all
$\theta\in\R$, where $\norm{\cdot}_1$ denotes the trace norm.
When this lower bound holds, we construct a commutator representation
directly. Otherwise, the vector-selection theorem of Marcus, Spielman,
and Srivastava yields smaller trace-zero compressions whose norms are
small enough for the induction to close.

We also construct an explicit family of zero-diagonal
Hermitian unitaries that forces a lower bound of order
$\sqrt{\log n}$ for $\norm{B}\norm{C}$ when either factor
is required to be diagonal in the prescribed basis.
The same family admits $\varepsilon$-pavings with fewer
than $2\varepsilon^{-2}$ blocks and representations by two
normal factors with optimal norm product $1/2$.
This establishes a distinction between unrestricted
commutator bounds and bounds under a prescribed diagonal
restriction.

The main results and their essential inputs are formalized
in Lean~4 using Mathlib. The development also includes a
formal derivation of the Kadison–Singer state-extension
theorem from the same vector-selection theorem.
\end{abstract}

\section{Introduction and main results}

For $A\in M_n(\C)$, let $\norm A$ denote its operator norm
on $\C^n$, and write $\comm BC=BC-CB$.
Throughout, the trace is unnormalized.
We write $\norm{H}_1=\tr\bigl((H^*H)^{1/2}\bigr)$
for the trace norm.
By the prescribed basis, we mean the standard coordinate
basis, fixed before the input matrix is chosen; all prescribed
diagonal restrictions and coordinate partitions refer to this basis.

Johnson, Ozawa, and Schechtman asked whether every trace-zero
matrix $A\in M_n(\C)$ admits a representation $A=[B,C]$ with
$B,C\in M_n(\C)$ and
$\norm{B}\norm{C}\le K\norm{A}$ for an absolute constant $K$.
They established the bound
$K_\varepsilon n^\varepsilon\norm{A}$ for every $\varepsilon>0$,
with $K_\varepsilon$ depending only on $\varepsilon$
\cite[Theorem~1]{johnson2013quantitative}.
Ravichandran and Srivastava subsequently obtained the explicit
bound $300e^{9\sqrt{\log n}}\norm{A}$
\cite[Theorem~6.3]{ravichandran2021asymptotically}.
Although these bounds grow more slowly than any fixed positive
power of $n$, they leave open whether dimension dependence
is necessary.

Theorem~\ref{thm:main} resolves this question for unrestricted
commutator factors: a single absolute constant works in every
dimension, with both factors belonging to the original matrix
algebra $M_n(\C)$. Thus the result removes dimension dependence
entirely and gives an affirmative answer to the uniform
commutator problem.

\begin{theorem}\label{thm:main}
There exists an absolute constant $K>0$ such that, for every
integer $n\ge1$ and every $A\in M_n(\C)$ with $\tr A=0$,
there exist $B,C\in M_n(\C)$ satisfying
\begin{equation}\label{eq:main}
 A=\comm BC,\qquad \norm B\norm C\le K\norm A.
\end{equation}
\end{theorem}

 The choice of norms is essential. Angel and
Schechtman~\cite{angel2015hilbert} proved that the optimal
bound for $\norm B\,\|C\|_{\mathrm F}$ relative to
$\|A\|_{\mathrm F}$ grows like $\sqrt{\log n}$, where
$\|\cdot\|_{\mathrm F}$  denotes the Frobenius norm.
They established an upper bound of this order for every
trace-zero matrix and a matching lower bound: for each
$n\ge2$, there is a nonzero trace-zero matrix
$A\in M_n(\C)$ such that every representation $A=[B,C]$
satisfies
\[
 \norm B\,\|C\|_{\mathrm F}
 \ge c\sqrt{\log n}\,\|A\|_{\mathrm F},
\]
with an absolute constant $c>0$.
Thus, even with unrestricted commutator factors, no
dimension-independent constant $K$ can guarantee
$\norm B\,\|C\|_{\mathrm F}\le K\|A\|_{\mathrm F}$
for every trace-zero matrix.




To state our second main result, we recall the diagonal commutator
quantities of Johnson, Ozawa, and
Schechtman~\cite{johnson2013quantitative}.
Let $Q=[-1,1]+\ii[-1,1]$. For a zero-diagonal
$A\in M_m(\C)$, define
\begin{equation}\label{eq:lambda_A}
\lambda(A)=
\min\bigl\{\norm C:
A=\comm BC,\;
B\text{ is diagonal with entries in }Q\bigr\},
\end{equation}
where the minimum is over $B,C\in M_m(\C)$.
For $m\ge2$, set
\begin{equation}\label{eq:lambda_m}
\lambda(m)=
\max\bigl\{\lambda(A):
A\in M_m(\C),\;
A\text{ has zero diagonal},\;
\norm A=1\bigr\},
\end{equation}
and put $\lambda(1)=0$.
Thus $\lambda(m)$ measures the largest normalized diagonal
commutator cost in dimension $m$.

For a coordinate set $S$, let $A_S$ denote the corresponding
principal submatrix. For $1\le r\le m$, define
\begin{equation}\label{eq:PrA}
p_r(A)=
\min_{\substack{\mathcal P\text{ a partition of }\{1,\ldots,m\}\\
                |\mathcal P|\le r}}
\max_{S\in\mathcal P}\norm{A_S},
\end{equation}
the smallest possible maximum compression norm among
pavings with at most $r$ blocks.
An $\varepsilon$-paving of $A$ is a coordinate partition
$\mathcal P$ such that $\max_{S\in\mathcal P}\norm{A_S}\le
\varepsilon\norm{A}$. For the norm-one matrices below, this means
$\max_{S\in\mathcal P}\norm{A_S}\le\varepsilon$.

The following family has uniformly bounded paving and
normal-factor commutator costs, while its diagonal commutator cost
grows with the dimension.

\begin{theorem}
\label{thm:main2}
For every power of two $n\ge2$, there is an explicit zero-diagonal Hermitian
unitary $A_n\in M_n(\C)$ such that:
\begin{enumerate}[label=(\roman*)]
\item For every dyadic $r\le n$,
\[
 p_r(A_n)=\sqrt{\frac{n/r-1}{n-1}}.
\]
In particular, for every $0<\varepsilon<1$, the entire family admits
$\varepsilon$-pavings with fewer than $2\varepsilon^{-2}$ blocks.
\item Among the representations $A_n=\comm BC$ with both factors normal, and
with no requirement that either be diagonal in the given basis, the least value
of $\norm B\norm C$ is exactly $1/2$.
\item If $n=4^k$ with $k\ge1$, then
\[
 \lambda(A_n)\ge\sqrt{\frac{(3k-1)n+1}{32(n-1)}}\ge\frac{\sqrt k}4.
\]
\end{enumerate}
\end{theorem}

Since $A_n$ has zero diagonal and norm one, part~(iii) gives
\begin{equation}\label{eq:lambdaintro}
\lambda(4^k)\ge\lambda(A_{4^k})\ge\frac{\sqrt{k}}4
\qquad(k\ge1).
\end{equation}
Hence $\sup_m\lambda(m)=\infty$.
For any representation $A_n=\comm BC$ with $B$ diagonal,
the entries of $B/\norm B$ lie in $Q$, so normalization
gives $\lambda(A_n)\le\norm B\norm C$.
The same conclusion holds when $C$ is diagonal, since
$[B,C]=[C,-B]$.
Thus the example forces dimension dependence for the
operator-norm product whenever either factor is diagonal
in the prescribed basis, while retaining the paving and
normal-factor bounds in parts~(i) and~(ii).

Claim~3 of~\cite{johnson2013quantitative} shows that a
uniform bound for $\lambda(m)$ would imply paving.
Our example rules out this hypothesis, and hence this
approach to proving paving; the conditional implication
itself remains valid.

Our third main contribution is a Lean~4 formalization
 built on Mathlib
of Theorems~\ref{thm:main} and~\ref{thm:main2} and their
essential inputs~\cite{moura2021lean,mathlib}.
In particular, we formalize the quantitative theorem of
Marcus, Spielman, and
Srivastava~\cite[Theorem~1.4]{marcus2015interlacing},
which is essential to our proof of Theorem~\ref{thm:main}.
From the same result, we also formally derive the
Kadison--Singer theorem: every pure state on the bounded
diagonal algebra has a unique state extension to
$B(\ell^2(\mathbb N;\mathbb C))$.
The development also includes the subpolynomial commutator
bounds of Johnson, Ozawa, and Schechtman.

The formal proofs are checked by Lean's kernel and depend
only on \texttt{propext}, \texttt{Classical.choice}, and
\texttt{Quot.sound}, with no additional axioms or
\texttt{sorryAx} dependencies.
The code is publicly available.\footnote{
\url{https://github.com/WuProver/lean-commutator}}

\subsection{Proof strategy and outline}\label{sec:outline}


The proof of Theorem~\ref{thm:main} combines a direct
construction with strong induction on the dimension.
The two cases are developed in Section~\ref{sec:dichotomy}.
For matrices in $\mathcal T_n(2^{-25})$, the uniform
trace-norm lower bound in
Definition~\ref{def:rotated-lower-bound} yields an
invertible off-diagonal block through a numerical-range
argument. A similarity transforms this block into the
identity, allowing us to apply the fixed-point construction
and elementary similarities of Section~\ref{sec:riccati}.
Outside this class, we use the vector-selection theorem
of Marcus, Spielman, and Srivastava to obtain smaller
trace-zero compressions. Their norm reduction offsets
the cost of assembling their commutator representations,
closing the induction in Section~\ref{sec:proofs}.
Section~\ref{subsec:diagonal-commutator-separation}
constructs the skew-Hadamard family in
Theorem~\ref{thm:main2} and proves its three properties.
Section~\ref{sec:formal} presents the formalized results
and their Lean declarations.

\section{Commutator representations with an identity block}
\label{sec:riccati}

This section provides the construction used in
Section~\ref{sec:high} when the uniform lower bound holds. We first obtain a
commutator bound for a two-block matrix whose upper-right
block is the identity. We then incorporate additional blocks.
The bounds depend only on the matrix norm and the number
of blocks, independently of their dimensions.

All spaces in the matrix arguments are finite-dimensional
complex Hilbert spaces, and rectangular blocks carry their operator norms.

\subsection{An elementary two-commutator decomposition}
\label{sec:sumtwo}

The two-block construction uses a decomposition of a
trace-zero matrix into two commutators with controlled factor norms.
We first obtain a single commutator representation for a Hermitian matrix.

\begin{lemma}\label{lem:hermitian}
Let $m\ge1$. If $G\in M_m(\C)$ is Hermitian and
$\tr G=0$, there are a unitary $U$ and a matrix $V$ such that
\[
G=\comm UV,\qquad \norm{V}\le\norm{G}.
\]
\end{lemma}

\begin{proof}
For $m=1$, take $U=I$ and $V=0$. Otherwise, diagonalize
$G$ unitarily. Its eigenvalues lie in $[-a,a]$, where
$a=\norm{G}$, and sum to zero.

Order the eigenvalues so that every partial sum
$s_j=\sum_{i=1}^j\lambda_i$ lies in $[-a,a]$.
Such an ordering is obtained by choosing a remaining
eigenvalue of opposite sign to the current sum, or any
remaining eigenvalue when the sum is zero. The required
sign is always available because the remaining eigenvalues
sum to the negative of the current sum.

In this ordered eigenbasis, set
\[
U=\sum_{j=1}^{m-1}E_{j,j+1}+E_{m,1},
\qquad
V=\sum_{j=1}^{m-1}s_jE_{j+1,j}.
\]
Then $U$ is unitary and
$\norm{V}=\max_{j<m}|s_j|\le a$. Moreover,
\[
UV=\diag(s_1,\ldots,s_{m-1},0),
\qquad
VU=\diag(0,s_1,\ldots,s_{m-1}).
\]
Since $s_m=0$, their difference is
$\diag(\lambda_1,\ldots,\lambda_m)$.
Conjugating both factors back proves the claim.
\end{proof}

For a matrix $H$, write
$\Rea H=(H+H^*)/2$ and $\Ima H=(H-H^*)/(2\ii)$.
Applying the preceding lemma to these two Hermitian parts gives the
decomposition needed below.

\begin{corollary}\label{cor:sumtwo}
Let $m\ge1$. Every trace-zero $H\in M_m(\C)$ admits
a representation
\begin{equation}\label{eq:sumtwo}
H=\comm UV+\comm KT,\qquad
U,K\text{ unitary},\qquad
\max(\norm{V},\norm{T})\le\norm{H}.
\end{equation}
\end{corollary}

\begin{proof}
Write $H=H_1+\ii H_2$ with $H_1=\Rea H$ and
$H_2=\Ima H$. Both matrices are Hermitian, have trace zero, and have
norms at most $\norm{H}$.
Apply Lemma~\ref{lem:hermitian} to obtain
$H_1=\comm UV$ and $H_2=\comm KW$, and set $T=\ii W$.
\end{proof}

The factors in \eqref{eq:sumtwo} may depend on $H$.
They are chosen once and remain fixed throughout the
construction below.

\subsection{A commutator bound for two blocks}

We now use the two-commutator decomposition to construct
a single commutator for a two-block matrix whose upper-right block is
the identity. A fixed-point argument will control the norms of its factors.

For $a\ge1$, define
\begin{equation}\label{eq:P}
\Pi(a)=(5a+2)^4(70a^2+62a+9)
     (140a^3+124a^2+22a+2).
\end{equation}
This bound depends only on $a$, independently of the
block size.

\begin{lemma}\label{prop:core}
Let $m\ge1$, $a\ge1$, and
\[
M=\begin{pmatrix}A&I_m\\F&D\end{pmatrix},
\qquad
\tr M=0,\qquad \norm{M}\le a.
\]
Then there exist $B,C\in M_{2m}(\C)$ such that
\[
M=\comm BC,\qquad \norm{B}\norm{C}\le\Pi(a).
\]
\end{lemma}

\begin{proof}
Since $\tr(A+D)=0$ and $\norm{A+D}\le2a$,
Corollary~\ref{cor:sumtwo} gives fixed factors satisfying
\[
A+D=\comm UV+\comm KT,\qquad
\norm{U}=\norm{K}=1,\qquad
\norm{V},\norm{T}\le2a.
\]
Put $\rho=70a^2+62a+6$ and define
\begin{equation}\label{eq:defects}
\begin{aligned}
\widetilde U&=\rho I+U,
& D_0&=\comm UV-A,\\
D_1&=\comm KV+D_0U-K+\comm U{TK}-F.
\end{aligned}
\end{equation}
The block norm bounds give
\[
\norm{D_0}\le5a,\qquad
\norm{D_1}\le4a+5a+1+4a+a=14a+1.
\]

For a matrix $Q$ to be chosen, set
\begin{equation}\label{eq:companions}
\begin{gathered}
B_c=\begin{pmatrix}\widetilde U&I\\K&0\end{pmatrix},
\qquad
S=\begin{pmatrix}I&0\\Q&I\end{pmatrix},\\
C_c=
\begin{pmatrix}
V&T\\
Q-D_0+TK&V+I-\widetilde U T
\end{pmatrix}.
\end{gathered}
\end{equation}
Direct block multiplication gives
\begin{equation}\label{eq:fourblocks}
\begin{aligned}
\comm{B_c}{C_c}
&=
\begin{pmatrix}
A+Q&I\\
F+\rho D_0+D_1-Q\widetilde U&D-Q
\end{pmatrix},\\
S^{-1}MS
&=
\begin{pmatrix}
A+Q&I\\
F+DQ-QA-Q^2&D-Q
\end{pmatrix}.
\end{aligned}
\end{equation}
Thus the two matrices agree precisely when
\[
\rho(Q-D_0)=Q^2-DQ+Q(A-U)+D_1.
\]
We solve this equation by finding a fixed point of
\begin{equation}\label{eq:fixedpoint}
\Psi(Q)=D_0+\frac1{\rho}
\bigl(Q^2-DQ+Q(A-U)+D_1\bigr)
\end{equation}
in the closed ball $\norm{Q-D_0}\le1$.

On this ball, $\norm{Q}\le5a+1$, and hence
\[
\norm{\Psi(Q)-D_0}
\le
\frac{(5a+1)^2+(2a+1)(5a+1)+14a+1}{\rho}
=\frac12.
\]
For two matrices $Q,Q'$ in the ball, the identity
\[
Q^2-Q'^2=(Q-Q')Q+Q'(Q-Q')
\]
also gives
\[
\norm{\Psi(Q)-\Psi(Q')}
\le\frac{12a+3}{\rho}\norm{Q-Q'}
\le\frac12\norm{Q-Q'}.
\]
The Banach fixed-point theorem therefore supplies a solution
with $\norm{Q-D_0}\le1$. Equation~\eqref{eq:fourblocks}
then yields
\begin{equation}\label{eq:3_6}
M=\comm{SB_cS^{-1}}{SC_cS^{-1}}.
\end{equation}

It remains to estimate the factors. Summing block norms,
we obtain
\[
\norm{B_c}\le\rho+3.
\]
The lower-left block of $C_c$ has norm at most $1+2a$,
and its lower-right block has norm at most
$1+2a(\rho+2)$. Therefore
\[
\norm{C_c}
\le2a+2a+(1+2a)+(1+2a(\rho+2))
=2a(\rho+5)+2.
\]
Finally,
\[
\norm{S},\norm{S^{-1}}
\le1+\norm{Q}\le5a+2.
\]
The norm product in \eqref{eq:3_6} is consequently at most
\[
(5a+2)^4(\rho+3)\bigl(2a(\rho+5)+2\bigr)=\Pi(a),
\]
as required.
\end{proof}

\subsection{Incorporating the remaining blocks}
\label{sec:shears}

We now use similarities that preserve the identity block
to make all remaining diagonal blocks zero. The first two
blocks then form a trace-zero matrix to which
Lemma~\ref{prop:core} applies. The off-diagonal blocks are
recovered using the following elementary estimate.

\begin{lemma}
\label{lem:sylvester}
Let $U$ and $V$ act on possibly different Hilbert spaces,
and let $z,w\in\C$. If
$\norm{U}+\norm{V}<|z-w|$, then, for every
rectangular matrix $Y$, the equation
\[
(z-w)X+UX-XV=Y
\]
has a solution satisfying
\[
\norm{X}\le\frac{\norm{Y}}{|z-w|-\norm{U}-\norm{V}}.
\]
\end{lemma}

\begin{proof}
Put $\eta=\norm{U}+\norm{V}$.
On the space of rectangular matrices with its operator norm,
the map $\mathcal E(X)=UX-XV$ has norm at most $\eta$.
The inverse of $(z-w)I+\mathcal E$ is given by the
norm-convergent Neumann series
\[
\frac1{z-w}\sum_{j=0}^{\infty}
\left(-\frac{\mathcal E}{z-w}\right)^j.
\]
Its norm is at most $(|z-w|-\eta)^{-1}$.
\end{proof}

The Sylvester estimate bounds each off-diagonal block separately.
The next lemma combines these bounds into an estimate for the whole
operator, independently of the block dimensions.

\begin{lemma}\label{lem:offnorm}
Let $T=(T_{ij})$ act on an orthogonal direct sum of $k\ge1$ Hilbert
spaces, and suppose $T_{ii}=0$ for every $i$.
If $c\ge0$ and $\norm{T_{ij}}\le c$ whenever $i\ne j$, then
\begin{equation}\label{eq:offnorm}
\norm{T}\le c(k-1).
\end{equation}
\end{lemma}

\begin{proof}
The case $k=1$ is immediate. For $k\ge2$ and $x=(x_j)$,
the triangle inequality and Cauchy--Schwarz give
\[
\norm{Tx}^2
\le c^2(k-1)\sum_i\sum_{j\ne i}\norm{x_j}^2
=c^2(k-1)^2\norm{x}^2.
\]
Taking square roots proves the estimate.
\end{proof}

We can now eliminate the remaining diagonal blocks by similarities and
recover the off-diagonal blocks with the two preceding estimates.
For an invertible matrix $S$, write
$\chi(S)=\norm{S}\norm{S^{-1}}$; conjugating both commutator factors
by $S$ multiplies their norm product by at most $\chi(S)^2$.
The norm growth of each elimination step is bounded by the function
\begin{equation}\label{eq:iterate}
f(x)=x(x+2)^2\bigl(1+x(x+2)^2\bigr)^2\qquad(x\ge1).
\end{equation}

\begin{proposition}\label{prop:absorb}
Let a trace-zero matrix $A$ have an orthogonal block
decomposition of sizes $r,r,s_3,\ldots,s_b$, where
$b\ge2$, $r\ge1$, and $1\le s_j\le r$.
Suppose $A_{12}=I_r$ and $\norm{A}\le a$, with $a\ge1$.
Then there are matrices $B,C$ of the same size as $A$
such that $A=\comm BC$ and
\begin{equation}\label{eq:absorb}
\norm{B}\norm{C}
\le
\left(\frac{f^{\circ(b-2)}(a)}{a}\right)^2
(b-1)\Bigl(4\Pi\bigl(f^{\circ(b-2)}(a)\bigr)
+2(b-1)f^{\circ(b-2)}(a)\Bigr),
\end{equation}
where $f^{\circ(m)}$ denotes the composition of $f$ with itself
$m$ times, with $f^{\circ(0)}=\mathrm{id}$.
\end{proposition}

\begin{proof}
For this proof, put $\widehat a=f^{\circ(b-2)}(a)$.
Write $E_{ij}(X)$ for the matrix with $X$ in block
$(i,j)$ and zero elsewhere. For $i\ne j$,
\[
(I+E_{ij}(X))^{-1}=I-E_{ij}(X).
\]
These elementary similarities will eliminate the diagonal
blocks indexed by $j\ge3$.

Fix such a $j$ and an isometric embedding
$J:\C^{s_j}\to\C^r$, so that $J^*J=I$.
First set
\[
X=J-A_{1j},\qquad
S_1=I+E_{2j}(X),\qquad
A'=S_1^{-1}AS_1.
\]
Direct multiplication gives
\[
A'_{1j}=A_{1j}+A_{12}X=J,\qquad A'_{12}=I.
\]
Only diagonal blocks $2$ and $j$ can change.
Next set
\[
Y=A'_{jj}J^*,\qquad
S_2=I+E_{j1}(Y),\qquad
A''=S_2^{-1}A'S_2.
\]
This time only diagonal blocks $1$ and $j$ can change,
and
\begin{equation}\label{eq:shearzero}
A''_{jj}=A'_{jj}-YA'_{1j}=0,\qquad A''_{12}=I.
\end{equation}
Thus these two similarities eliminate block $j$ while
preserving the identity block and every previously
eliminated diagonal block.

We next bound the accumulated change of basis.
If the current matrix has norm at most $x\ge1$, then
$\norm{X}\le x+1$, so
\[
\chi(S_1)\le(x+2)^2,\qquad
\norm{A'}\le x(x+2)^2.
\]
Since $\norm{Y}\le x(x+2)^2$, we also have
\[
\chi(S_2)\le\bigl(1+x(x+2)^2\bigr)^2,\qquad
\norm{A''}\le f(x).
\]
A pair of these similarities therefore replaces the norm
bound $x$ by $f(x)$ and has condition number at most
$f(x)/x$.

The function $f$ is nondecreasing on $[1,\infty)$ and
satisfies $f(x)\ge x$. Applying the construction for
$j=3,\ldots,b$ and multiplying the successive condition
number bounds gives
\begin{equation}\label{eq:telescope}
\widehat A=S^{-1}AS,\qquad
\norm{\widehat A}\le\widehat a,\qquad
\chi(S)\le\widehat a/a.
\end{equation}
When $b=2$, no similarities are needed: take $S=I$.

Merge the first two blocks of $\widehat A$ into a single
block $M$. All other diagonal blocks are zero, so
$\tr M=0$. Its upper-right block is still the identity,
and $\norm{M}\le\widehat a$.
Lemma~\ref{prop:core} gives
\[
M=\comm{B_0}{C_0},
\qquad
\norm{B_0}\norm{C_0}\le\Pi(\widehat a).
\]
Since $M\ne0$, we may rescale the factors inversely so that
\[
\norm{B_0}=\frac14,\qquad
\norm{C_0}\le4\Pi(\widehat a).
\]

There are now $k=b-1$ blocks. Set
\[
\widehat B=\diag(B_0,I,2I,\ldots,(k-1)I),
\qquad
\widehat C_{11}=C_0,\quad
\widehat C_{jj}=0\quad(j>1).
\]
For $k=1$, these definitions simply give
$\widehat B=B_0$ and $\widehat C=C_0$.
For each pair $i\ne j$, choose $\widehat C_{ij}$ to solve
\[
\widehat B_{ii}\widehat C_{ij}
-\widehat C_{ij}\widehat B_{jj}
=\widehat A_{ij}.
\]
Lemma~\ref{lem:sylvester} applies with scalar centers
$0,1,\ldots,k-1$. Their separation is at least one,
and the only non-scalar perturbation is $B_0$,
whose norm is $1/4$. Thus
\[
\norm{\widehat C_{ij}}\le2\widehat a.
\]
By Lemma~\ref{lem:offnorm} and the block diagonal norm formula,
\[
\norm{\widehat B}\le k,\qquad
\norm{\widehat C}
\le4\Pi(\widehat a)+2(k-1)\widehat a.
\]
The diagonal and off-diagonal block equations now give
$\comm{\widehat B}{\widehat C}=\widehat A$.

Finally, conjugating both factors by $S$ gives a
commutator representation of $A$ and multiplies the
norm product by at most $\chi(S)^2$.
Equation~\eqref{eq:telescope} therefore yields
\eqref{eq:absorb}.
\end{proof}

\section{Rotated Hermitian parts and proof of the main theorem}
\label{sec:dichotomy}

For $A\in M_n(\C)$, we fix $0<t\le1$ and divide the proof according to
whether $A\in\mathcal T_n(t)$, the set introduced below. In this set, we construct a
commutator representation directly; otherwise, we obtain smaller
trace-zero compressions and apply induction. The compression argument
uses the following theorem of Marcus, Spielman, and
Srivastava~\cite[Theorem~1.4]{marcus2015interlacing}.

\begin{theorem}
\label{thm:MSS}
Let $x_1,\ldots,x_M$ be independent, finitely supported
random vectors in a finite-dimensional complex Hilbert space.
Suppose that, for some $\varepsilon>0$,
\begin{equation}\label{eq:MSSassumptions}
\sum_i\mathbb E x_ix_i^*=I,\qquad
\mathbb E\norm{x_i}^2\le\varepsilon.
\end{equation}
Then some outcome satisfies
\begin{equation}\label{eq:MSSconclusion}
\norm{\sum_i x_ix_i^*}\le(1+\sqrt{\varepsilon})^2.
\end{equation}
\end{theorem}

The covariance sums arising below may be bounded above
by the identity. The following corollary covers this case.

\begin{corollary}\label{cor:MSS-subidentity}
Under the hypotheses of Theorem~\ref{thm:MSS}, with
the covariance equality replaced by
$\sum_i\mathbb E x_ix_i^*\le I$, the conclusion
\eqref{eq:MSSconclusion} still holds.
\end{corollary}

\begin{proof}
Spectrally decompose the positive semidefinite
difference $I-\sum_i\mathbb E x_ix_i^*$ and split its eigenvalues into
pieces of size at most $\varepsilon$. Adding the corresponding
deterministic vectors gives equality in \eqref{eq:MSSassumptions}
while preserving the individual expected squared-norm bounds.
Apply Theorem~\ref{thm:MSS} and then discard the added positive
semidefinite contributions.
\end{proof}

A compression of $A$ to a subspace $V$ means $P_VA|_V$, where $P_V$
is the orthogonal projection onto $V$; invariance of $V$ is not assumed.
To explain its use, suppose $\norm{A}=1$ and write
\[
H_\theta=\Rea(e^{\ii\theta}A),\qquad
E=(H_\theta^2)^{1/2}.
\]
Since $-E\le H_\theta\le E$, controlling a compression
of $E$ also controls the corresponding compression of
$H_\theta$. For an orthonormal basis $(v_i)$, the vectors
$f_i=E^{1/2}v_i$ satisfy
\[
\sum_i f_if_i^*=E\le I,\qquad
\norm{f_i}^2=v_i^*Ev_i,\qquad
\frac1n\sum_i\norm{f_i}^2=\frac{\tr E}{n}
=\frac{\norm{H_\theta}_1}{n}.
\]
Thus a small trace norm gives a small average squared
vector norm. Below, we choose the basis so that individual
bounds are available as well.

We use a short notation for the matrices whose rotated
Hermitian parts all satisfy a uniform trace-norm lower bound.

\begin{definition}
\label{def:rotated-lower-bound}
For an integer $n\ge1$ and $0<t\le1$, let $\mathcal T_n(t)$ be the
set of matrices $A\in M_n(\C)$ satisfying
\begin{equation}\label{eq:rotated-lower-bound}
\norm{\Rea(e^{\ii\theta}A)}_1
\ge tn\norm{A}
\qquad\text{for every }\theta\in\R.
\end{equation}
We say that such a matrix has a uniform trace-norm
lower bound for rotated Hermitian parts with parameter $t$.
\end{definition}

For fixed $t$ and $A\in M_n(\C)$, either $A\in\mathcal T_n(t)$ or
$\norm{\Rea(e^{\ii\theta}A)}_1<tn\norm{A}$ for some $\theta$.
The proof of the main theorem uses $t=2^{-25}$, independently of
the matrix and its dimension.

\subsection{When the uniform lower bound holds}
\label{sec:high}

The uniform lower bound first yields a disk in the
numerical range of every compression of sufficiently small codimension.
We will use this to construct an invertible off-diagonal block of
proportional rank and apply Section~\ref{sec:riccati}.

\begin{lemma}
\label{lem:disk}
Suppose $0<t\le1$, $\tr A=0$, $\norm{A}=1$, and
$A\in\mathcal T_n(t)$.
If $W$ has codimension
$d\le tn/4$, then the numerical range of
$T=P_WA|_W$ contains the closed disk centered at zero
with radius $t/4$.
\end{lemma}

\begin{proof}
For every $\theta$, the Hermitian matrix
$H_\theta=\Rea(e^{\ii\theta}A)$ has norm at most one
and trace zero. Its positive eigenvalues therefore sum to
\[
\frac12\norm{H_\theta}_1\ge tn/2.
\]
Order its eigenvalues decreasingly. Each is at most one,
so fewer than $d+1$ positive eigenvalues would give a positive sum
at most $d<tn/2$, a contradiction. Bounding the first $d$ eigenvalues
by one and each remaining positive eigenvalue by
$\lambda_{d+1}(H_\theta)$ gives
\[
tn/2\le d+(n-d)\lambda_{d+1}(H_\theta).
\]
Consequently,
\[
\lambda_{d+1}(H_\theta)
\ge\frac{tn/2-d}{n-d}\ge t/4.
\]
The min--max principle gives
\[
\lambda_{\max}(P_WH_\theta|_W)\ge t/4.
\]
These largest eigenvalues are the directional supports of
the numerical range of $T$. That range is compact and
convex, so it contains the disk of radius $t/4$;
see~\cite[Section~2]{li2008canonical}.
\end{proof}

To construct the off-diagonal block, we need a unit
vector with zero expectation and a large image under the compressed
operator. The following lemma extracts such a vector from two opposite
points in its numerical range.

\begin{lemma}
\label{lem:neutral}
If the numerical range of $T$ contains both $\rho$ and
$-\rho$, where $\rho>0$, then there is a unit vector $v$
such that
\begin{equation}\label{eq:neutral}
v^*Tv=0,\qquad \norm{Tv}\ge\rho.
\end{equation}
\end{lemma}

\begin{proof}
Choose unit vectors $x_+,x_-$ with expectations
$\rho,-\rho$. Maximize $\tr(T^*TD)$ over matrices satisfying
\[
D\ge0,\qquad \tr D=1,\qquad \tr(TD)=0.
\]
The feasible set is compact and contains
$(x_+x_+^*+x_-x_-^*)/2$, whose objective value is at least
$\rho^2$ by Cauchy--Schwarz.

A maximizer can be chosen to be an extreme point of the
feasible set. Such a point has rank one. Indeed, if its
rank were $s\ge2$, the Hermitian operators on its support
would form a real vector space of dimension $s^2\ge4$.
The three real constraints above would admit a nonzero
supported Hermitian perturbation $E$ in their common
kernel. For sufficiently small $\eta>0$, both
$D+\eta E$ and $D-\eta E$ would remain feasible,
contradicting extremality.

The maximizing matrix is therefore $vv^*$ for a unit
vector $v$. Its constraints and objective value give
$v^*Tv=0$ and $\norm{Tv}^2\ge\rho^2$.
\end{proof}

Repeated application of the two lemmas produces an invertible
off-diagonal block whose rank is proportional to the dimension.
Making that block the identity allows us to apply
Proposition~\ref{prop:absorb}, with the following bound for $0<t\le1$:
\begin{equation}\label{eq:KH}
\begin{aligned}
K_H(t)={}&\bigl[f^{\circ\lceil16/t\rceil}(4/t)\bigr]^2
\bigl(\lceil16/t\rceil+1\bigr)\\
&\quad\times\Bigl[4\Pi\bigl(f^{\circ\lceil16/t\rceil}(4/t)\bigr)
+2\bigl(\lceil16/t\rceil+1\bigr)
f^{\circ\lceil16/t\rceil}(4/t)\Bigr].
\end{aligned}
\end{equation}

\begin{proposition}\label{prop:high}
Let $n\ge1$ be an integer and $0<t\le1$. Every trace-zero matrix
$A\in\mathcal T_n(t)$ admits a
representation $A=\comm BC$, with $B,C$ of the same size as $A$, such that
\[
\norm{B}\norm{C}\le K_H(t)\norm{A}.
\]
\end{proposition}

\begin{proof}
For this proof, put $L=\lceil16/t\rceil$,
$a_*=f^{\circ L}(4/t)$, and $k_*=L+1$.
The zero matrix is immediate. Normalize $\norm{A}=1$
and set $\rho=t/4$ and $r=\lceil tn/16\rceil$.
We construct orthonormal vectors $v_1,\ldots,v_r$ such that
\[
v_i^*Av_j=0\quad\text{for all }i,j,
\]
while the vectors $Av_i$ are pairwise orthogonal and have
norms at least $\rho$.

After choosing $v_1,\ldots,v_j$, work in
\[
W_j=
\Span\{v_i,Av_i,A^*v_i,A^*Av_i:1\le i\le j\}^{\perp}.
\]
For $0\le j\le r-1$,
\[
\codim W_j\le4j<tn/4.
\]
Lemmas~\ref{lem:disk} and~\ref{lem:neutral}, applied to
$P_{W_j}A|_{W_j}$, give a unit vector $v_{j+1}\in W_j$
with
\[
v_{j+1}^*Av_{j+1}=0,\qquad
\norm{Av_{j+1}}
\ge\norm{P_{W_j}Av_{j+1}}\ge\rho.
\]
The defining orthogonality conditions preserve
orthonormality, make both mixed coefficients
$v_i^*Av_{j+1}$ and $v_{j+1}^*Av_i$ vanish, and ensure
$Av_{j+1}\perp Av_i$.

Let $V=\Span\{v_i\}$ and $W=A(V)$. Then $V\perp W$
and both spaces have dimension $r$, so $2r\le n$.
In the bases
\[
w_i=\frac{Av_i}{\norm{Av_i}}
\quad\text{of }W,\qquad
v_i\quad\text{of }V,
\]
the $(1,2)$ block of $A$ is
\[
D=\diag(\norm{Av_1},\ldots,\norm{Av_r}),
\qquad \rho I\le D\le I.
\]
The similarity
\[
S_0=\diag(I_r,D^{-1},I_{\rm rest})
\]
has condition number at most $4/t$.
It makes the $(1,2)$ block of $S_0^{-1}AS_0$ equal to
$I_r$ and gives $\norm{S_0^{-1}AS_0}\le4/t$.

Partition the remaining space into blocks of positive
size at most $r$. Since $r\ge tn/16$, there are at most
$L$ such blocks. Proposition~\ref{prop:absorb} applies
with initial norm bound $4/t$. By monotonicity of $f$
and $\Pi$, its commutator bound is at most
\[
\left(\frac{a_*}{4/t}\right)^2
k_*\bigl(4\Pi(a_*)+2k_*a_*\bigr).
\]
Conjugating back by $S_0$ multiplies this bound by at most
$(4/t)^2$, giving \eqref{eq:KH}.
Finally, undo the scalar normalization.
\end{proof}

\subsection{When the uniform lower bound fails}
\label{sec:low}

If $A\in M_n(\C)$ and $A\notin\mathcal T_n(t)$ at the chosen parameter,
some rotated Hermitian part has trace norm below $tn\norm{A}$.
We use this direction to construct equal-rank trace-zero compressions
with small operator norms. Their subspaces may depend on $A$ and
need not be spanned by vectors from a prescribed
coordinate basis.

We begin with two tools for grouping the spectral data
and choosing a basis within each group.
We use the following consequence of the Steinitz theorem:
vectors $u_1,\ldots,u_N\in\R^2$ with sum zero and
$\norm{u_i}_\infty\le1$ can be reordered so that every
partial sum has $\ell_\infty$ norm at most $2$.
This is the case $d=2$
of~\cite[Theorem~1.2]{barany2024matrix}.

We also use the simultaneous hollowization theorem of
Damm and Fa\ss bender: three trace-zero Hermitian matrices
can be conjugated by one unitary so that the first two
have zero diagonal and the third has at most two nonzero
diagonal entries
\cite[Proposition~2.13(b)]{damm2020simultaneous}.

Subtracting the scalar means gives the following choice
of basis, which preserves two diagonal expectations and bounds a third.

\begin{lemma}\label{lem:flatbasis}
Let $H,G$ be Hermitian operators on $\C^R$, with $R\ge2$,
and let $E\ge0$. There is an orthonormal basis
$v_1,\ldots,v_R$ such that
\begin{equation}\label{eq:flatbasis}
v_a^*Hv_a=\frac{\tr H}{R},\qquad
v_a^*Gv_a=\frac{\tr G}{R},\qquad
v_a^*Ev_a\le\frac{2\tr E}{R}
\quad(1\le a\le R).
\end{equation}
\end{lemma}

\begin{proof}
Subtract the scalar mean from each matrix and apply
the simultaneous hollowization theorem.
The resulting diagonal entries of $H$ and $G$ equal
their respective means.
All but at most two diagonal entries of $E$ equal
$\tr E/R$. The remaining entries are nonnegative and
have total at most $2\tr E/R$, proving the last bound.
\end{proof}

To preserve the trace identities supplied by the
preceding lemma, each selected set must contain exactly one index
from every group. The following consequence of Corollary~\ref{cor:MSS-subidentity}
enforces this requirement while controlling the covariance sums.

\begin{lemma}\label{lem:transversals}
Let $\kappa\ge0$ and let $R\ge1$ be a power of two.
Suppose vectors are indexed by finitely many groups of
size $R$ and satisfy
\[
\sum_i f_if_i^*\le I,\qquad
\norm{f_i}^2\le\kappa/R.
\]
Their indices can be partitioned into $R$ sets, each
containing exactly one index from every group, such that
each resulting set $J$ satisfies
\begin{equation}\label{eq:transversalbound}
\norm{\sum_{i\in J}f_if_i^*}
\le\frac{\bigl[1+2(\sqrt2+1)\sqrt{\kappa}\bigr]^2}{R}.
\end{equation}
\end{lemma}

\begin{proof}
Write $R=2^h$ with $h\ge0$ an integer.
If $\kappa=0$, all vectors vanish and any such partition works.
Assume $\kappa>0$ and put $\delta=\kappa/R$.

Consider a set containing the same even number of indices
from each group, with $\sum_i f_if_i^*\le\Lambda I$.
Pair the indices within each group. For each pair of
vectors $a,b$, independently choose one of
\begin{equation}\label{eq:fourchoices}
\sqrt2(a\oplus b),\quad
\sqrt2(a\oplus(-b)),\quad
\sqrt2(b\oplus a),\quad
\sqrt2(b\oplus(-a))
\end{equation}
with equal probability. Denote the chosen random vector
by $w$. The expected off-diagonal covariance blocks cancel,
and
\[
\sum\mathbb E ww^*
=
\diag\left(\sum_i f_if_i^*,\sum_i f_if_i^*\right)
\le\Lambda I,\qquad
\mathbb E\norm{w}^2\le4\delta.
\]
For $\Lambda>0$, apply Corollary~\ref{cor:MSS-subidentity}
to $w/\sqrt{\Lambda}$, with
$\varepsilon=4\delta/\Lambda$. Some outcome satisfies
\[
\norm{\sum ww^*}\le(\sqrt{\Lambda}+2\sqrt\delta)^2.
\]
Its two diagonal blocks are twice the corresponding
vector sums for two complementary sets, each containing
one index from every pair. Thus both sets satisfy the bound
\begin{equation}\label{eq:binaryrecurrence}
\Lambda_{\rm next}I,\qquad
\Lambda_{\rm next}=\frac12(\sqrt{\Lambda}+2\sqrt\delta)^2.
\end{equation}
If $\Lambda=0$, all vectors in the set vanish and can be
divided arbitrarily.

Starting with $\Lambda_0=1$, repeat this division for $h$ levels.
Each final set contains exactly one index from every
original group, and these sets partition all indices.
Taking square roots in \eqref{eq:binaryrecurrence} gives
\[
\sqrt{\Lambda_h}
\le2^{-h/2}
+\sqrt{2\delta}\sum_{j=0}^{h-1}2^{-j/2}
\le
\frac{1+2(\sqrt2+1)\sqrt{\kappa}}{\sqrt R}.
\]
Squaring proves \eqref{eq:transversalbound}.
\end{proof}

We now combine spectral grouping, the choice of basis,
and the preceding selection estimate. Together they turn a small
trace norm in one direction into equal-rank compressions with zero
trace and small operator norm.

\begin{proposition}\label{prop:low}
Let $R\ge2$ be a power of two and let $k\ge1$ be an integer.
Suppose $A\in M_{kR}(\C)$ has trace zero and, for some $\theta\in\R$,
\begin{equation}\label{eq:lowmass}
\norm{\Rea(e^{\ii\theta}A)}_1\le2k\norm{A}.
\end{equation}
There are $R$ mutually orthogonal rank-$k$ projections
$P_\ell$ summing to $I$ such that
\begin{equation}\label{eq:lowconclusion}
\tr(P_\ell AP_\ell)=0,\qquad
\norm{P_\ell AP_\ell}
\le\frac{319}{R}\norm{A}.
\end{equation}
\end{proposition}

When $R=2^{26}$, the matrix dimension is $n=k2^{26}$,
so $2^{-25}n\norm{A}=2k\norm{A}$. Thus
$A\notin\mathcal T_n(2^{-25})$ implies \eqref{eq:lowmass}, with a strict
inequality; the proposition also allows equality.

\begin{proof}
Write $n=kR$ and, for the estimates below, put
$C_{12}=\bigl[1+2(\sqrt2+1)\sqrt{12}\bigr]^2$.
The case $A=0$ is immediate. Normalize and rotate so that
$\norm{A}=1$ and $A=H+\ii G$, where $H,G$ are
trace-zero Hermitian matrices and
\[
E=(H^2)^{1/2}\le I,\qquad \tr E=\norm{H}_1\le2k.
\]

Choose an orthonormal eigenbasis $e_i$ of $G$, with
eigenvalues $\lambda_i$. The vectors
\[
\left(\lambda_i,e_i^*Ee_i-\frac{\tr E}{n}\right)
\in\R^2
\]
sum to zero and have $\ell_\infty$ norm at most one.
Apply the Steinitz rearrangement and divide the reordered
basis into $k$ consecutive groups of size $R$.
Each group sum is a difference of two partial sums.
Thus the corresponding spectral subspaces $V_j$ satisfy
\begin{equation}\label{eq:grouptraces}
|\tr G_j|\le4,\qquad
\tr E_j\le\frac{\tr E}{k}+4\le6,
\end{equation}
where $H_j,G_j,E_j$ denote compressions to $V_j$.
The spaces $V_j$ reduce $G$.

Apply Lemma~\ref{lem:flatbasis} within each $V_j$.
We obtain an orthonormal basis $v_{j,a}$ satisfying
\begin{equation}\label{eq:groupbasis}
v_{j,a}^*Hv_{j,a}=\frac{\tr H_j}{R},\qquad
v_{j,a}^*Gv_{j,a}=\frac{\tr G_j}{R},\qquad
v_{j,a}^*Ev_{j,a}\le\frac{12}{R}.
\end{equation}
The vectors $f_{j,a}=E^{1/2}v_{j,a}$ therefore satisfy
\[
\sum_{j,a}f_{j,a}f_{j,a}^*=E\le I,\qquad
\norm{f_{j,a}}^2\le12/R.
\]
Lemma~\ref{lem:transversals}, with $\kappa=12$, partitions
their indices into $R$ sets, each containing one index
from every group. Let $P_\ell$ project onto the span
of the corresponding basis vectors $v_{j,a}$.
These projections are mutually orthogonal, sum to $I$,
and have rank $k$.

The equal diagonal values in \eqref{eq:groupbasis} give
\[
\tr(P_\ell AP_\ell)
=\frac1R\sum_j(\tr H_j+\ii\tr G_j)
=\frac{\tr A}{R}=0.
\]
Since the spaces $V_j$ reduce $G$, the compression
$P_\ell GP_\ell$ is diagonal in its selected basis.
Equation~\eqref{eq:grouptraces} therefore gives
\[
\norm{P_\ell GP_\ell}\le4/R.
\]
Moreover,
\[
\norm{P_\ell EP_\ell}
=\norm{E^{1/2}P_\ell E^{1/2}}
=\norm{\sum_{(j,a)\in\ell}f_{j,a}f_{j,a}^*}
\le C_{12}/R.
\]
The first equality uses $\norm{X^*X}=\norm{XX^*}$ with
$X=E^{1/2}P_\ell$.
Since $-E\le H\le E$, it follows that
\[
\norm{P_\ell HP_\ell}\le C_{12}/R.
\]
Combining the two Hermitian parts,
\[
\norm{P_\ell AP_\ell}\le\frac{C_{12}+4}{R}<\frac{319}{R}.
\]
Indeed,
\[
C_{12}+4=149+96\sqrt2+8\sqrt6+8\sqrt3<319,
\]
using $\sqrt2<17/12$, $\sqrt6<5/2$, and $\sqrt3<7/4$.
Undoing the rotation and normalization proves
\eqref{eq:lowconclusion}.
\end{proof}

\subsection{Proof of the main theorem}
\label{sec:proofs}\label{sec:assembly}

We combine the two cases by strong induction on the
matrix dimension. The following lemma controls the norm
product when commutator representations of diagonal
blocks are assembled.

\begin{lemma}\label{lem:assembly}
Suppose $A$ is decomposed into
$1\le r\le2^{27}-1$ nonempty orthogonal blocks, and
each diagonal block has a representation
$A_{ii}=\comm{U_i}{V_i}$.
Then $A=\comm BC$ in the original dimension, with
\begin{equation}\label{eq:assembly}
\norm{B}\norm{C}
\le2^{16}\max_i\bigl(\norm{U_i}\norm{V_i}\bigr)+2^{42}\norm{A}.
\end{equation}
\end{lemma}

\begin{proof}
Put $p_i=\norm{U_i}\norm{V_i}$.
For each nonzero diagonal block, rescale its factors
inversely so that
$\norm{U_i}=1/4$ and $\norm{V_i}=4p_i$.
For a zero block, use two zero factors.

Choose distinct centers $z_i$ from the grid
\[
\left\{x+\ii y:
x,y\in
\left\{-2^{13}+\tfrac12,-2^{13}+\tfrac32,
\ldots,2^{13}-\tfrac12\right\}\right\}.
\]
It contains $2^{28}$ points with pairwise distances
at least one. Set
\[
B=\diag(z_iI+U_i),\qquad C_{ii}=V_i.
\]
Since $|z_i|\le\sqrt2(2^{13}-1/2)$, we have
$\norm{B}<2^{14}$.

For $i\ne j$, solve
\[
(z_i-z_j)C_{ij}+U_iC_{ij}-C_{ij}U_j=A_{ij}.
\]
Lemma~\ref{lem:sylvester} applies because
$\norm{U_i}+\norm{U_j}\le1/2$, and gives
\[
\norm{C_{ij}}\le2\norm{A_{ij}}\le2\norm{A}.
\]
The block equations yield $A=\comm BC$, while
Lemma~\ref{lem:offnorm} gives
\[
\norm{C}\le4\max_i p_i+2(r-1)\norm{A}.
\]
Multiplying by $2^{14}$ and using $r-1\le2^{27}$
proves \eqref{eq:assembly}.
\end{proof}

We apply the assembly lemma to the representations
obtained by testing membership in $\mathcal T_m(2^{-25})$ for a matrix of size $m$.
This parameter makes the compression estimate strong enough
to close the induction.

\begin{proof}[Proof of Theorem~\ref{thm:main}]
With $K_H$ defined by \eqref{eq:KH}, fix the finite
absolute constant
\begin{equation}\label{eq:K}
K=\max\{2^{16}K_H(2^{-25})+2^{42},\,2^{43}\}.
\end{equation}
We prove the bound by strong induction on $n$.
The zero matrix is immediate, which also covers $n=1$.
Let $A\ne0$ have trace zero and put $a=\norm{A}$.

Choose an orthonormal basis in which $A$ has zero
diagonal. Indeed, the average of its diagonal expectations
is zero, so convexity of the numerical range gives a
unit vector $v$ with $v^*Av=0$. The compression to
$v^\perp$ still has trace zero. Repeating this argument
produces the required basis; see
also~\cite[Corollary~2.14]{damm2020simultaneous}.

If $n<2^{26}$, apply Lemma~\ref{lem:assembly} to the
zero singleton diagonal blocks. The resulting norm
product is at most $2^{42}a\le Ka$.

Now suppose $n\ge2^{26}$ and write
\[
n=k2^{26}+q,\qquad k\ge1,\quad0\le q<2^{26}.
\]
Set $N=k2^{26}=n-q$.
Let $A_0$ be the compression to the span of the first
$N$ basis vectors, and retain the remaining
$q$ vectors as zero singleton diagonal blocks.
Then
\[
\tr A_0=0,\qquad \norm{A_0}\le a.
\]
The off-diagonal blocks are retained and will be handled
by Lemma~\ref{lem:assembly}.
If $A_0=0$, that lemma directly gives the required bound.
Otherwise, we distinguish two cases.

If $A_0\in\mathcal T_N(2^{-25})$,
Proposition~\ref{prop:high} gives a commutator
representation of $A_0$ with norm product at most
$K_H(2^{-25})\norm{A_0}$.
Assembling this block with the $q$ zero singletons gives
\[
\norm{B}\norm{C}
\le\bigl(2^{16}K_H(2^{-25})+2^{42}\bigr)a
\le Ka.
\]
The uniform lower bound here is measured relative to
$\norm{A_0}$. This case uses the direct construction,
including when $q=0$ and $A_0=A$.

If $A_0\notin\mathcal T_N(2^{-25})$, some $\theta$ satisfies
\[
\norm{\Rea(e^{\ii\theta}A_0)}_1
<2^{-25}N\norm{A_0}=2k\norm{A_0}.
\]
Proposition~\ref{prop:low}, with $R=2^{26}$, gives
$2^{26}$ trace-zero compressions $T_i$ of dimension
$k<n$ satisfying
\[
\norm{T_i}\le319\cdot2^{-26}a.
\]
By induction, each $T_i$ has a commutator representation
with norm product at most $319\cdot2^{-26}Ka$.
Together with the $q$ zero singletons, there are
\[
2^{26}+q\le2^{27}-1
\]
blocks. Lemma~\ref{lem:assembly} therefore gives
\[
\norm{B}\norm{C}
\le\left(\frac{319}{1024}K+2^{42}\right)a
\le Ka,
\]
because $319/1024<1/2$ and $K\ge2^{43}$.
This completes the induction.
\end{proof}

We have made no attempt to optimize $K$.
The choice $t=2^{-25}$ serves only to make the induction close:
the compression factor $319\cdot2^{-26}$ and the assembly factor
$2^{16}$ have product $319/1024<1/2$.
The dominant growth in the displayed bound comes from the repeated
similarities in Proposition~\ref{prop:absorb}. Each elimination step
replaces a norm bound $x$ by the degree-nine polynomial $f(x)$, and
\eqref{eq:KH} uses $\lceil16/t\rceil=2^{29}$ such iterations at the
chosen parameter. The resulting bound is explicit but extremely large.
Sharper control of these similarities, or a construction avoiding their
repeated norm growth, could substantially reduce the bound.

\section{Diagonal commutators and the obstruction to a uniform bound}
\label{subsec:diagonal-commutator-separation}
\label{sec:diagroute}

We prove Theorem~\ref{thm:main2} by constructing an explicit
family $(A_n)$ for which the dimension-independent bound
of Theorem~\ref{thm:main} fails when either commutator factor
is required to be diagonal in the prescribed basis.
For every $0<\varepsilon<1$, this family nevertheless admits
$\varepsilon$-pavings with fewer than $2\varepsilon^{-2}$
blocks, independently of the dimension. Each $A_n$ also
admits a representation $A_n=[B,C]$ with $B$ and $C$ normal
and $\norm{B}\norm{C}=1/2$.

Throughout this section, matrix norms are operator norms
in Euclidean space. All diagonal matrices, principal
submatrices, and coordinate partitions refer to the
prescribed coordinate basis. For convenience, we restate Theorem~\ref{thm:main2}
before proving it.

\begin{main2restatement}
For every power of two $n\ge2$, there is an explicit zero-diagonal Hermitian
unitary $A_n\in M_n(\C)$ such that:
\begin{enumerate}[label=(\roman*)]
\item For every dyadic $r\le n$,
\begin{equation}\label{eq:optimalpaving}
p_r(A_n)=\sqrt{\frac{n/r-1}{n-1}}.
\end{equation}
In particular, for every $0<\varepsilon<1$, the entire family admits $\varepsilon$-pavings with fewer than $2\varepsilon^{-2}$ blocks.
\item Among the representations $A_n=\comm BC$ with both factors normal,
and with no requirement that either be diagonal in the given basis, the least
value of $\norm B\norm C$ is exactly $1/2$.
\item If $n=4^k$ with $k\ge1$, then
\begin{equation}\label{eq:mainlower}
\lambda(A_n)\ge
\sqrt{\frac{(3k-1)n+1}{32(n-1)}}\ge\frac{\sqrt{k}}4.
\end{equation}
\end{enumerate}
\end{main2restatement}

We first construct the family and prove the diagonal lower
bound in part~(iii). We then establish the optimal paving
norms and the normal commutator cost in parts~(i) and~(ii).

\subsection{Construction of the family}

Start with $S_1=(0)$ and recursively define
\begin{equation}\label{eq:recursion}
S_{2n}=
\begin{pmatrix}
S_n&S_n+I_n\\
S_n-I_n&-S_n
\end{pmatrix}.
\end{equation}
This is the classical doubling construction for skew-Hadamard
matrices, written in terms of $S_n=H_n-I_n$;
see~\cite[Section~4.2, Theorem~4]{cati2024database}.
We verify the properties needed below.

\begin{lemma}\label{lem:construction}
For every power of two $n$, the matrix $S_n$ is real,
skew-symmetric, has zero diagonal and all off-diagonal
entries in $\{1,-1\}$, and satisfies
\begin{equation}\label{eq:square}
S_n^2=-(n-1)I_n.
\end{equation}
\end{lemma}

\begin{proof}
All assertions hold for $n=1$. The recursion preserves
skew symmetry and the stated entry conditions.
Since $S_n$ commutes with $S_n\pm I_n$, block multiplication
gives
\[
S_{2n}^2=
\begin{pmatrix}
2S_n^2-I_n&0\\
0&2S_n^2-I_n
\end{pmatrix}
=-(2n-1)I_{2n}.
\]
This completes the induction.
\end{proof}


For every power of two $n\ge 2$, set
\begin{equation}\label{eq:A}
A_n=\frac{\ii S_n}{\sqrt{n-1}}.
\end{equation}
Lemma~\ref{lem:construction} gives
$A_n^*=A_n$ and $A_n^2=I_n$, so $A_n$ is a Hermitian
unitary. Moreover,
\begin{equation}\label{eq:flat}
(A_n)_{ii}=0,\qquad
|(A_n)_{ij}|^2=\frac1{n-1}\quad(i\ne j).
\end{equation}
In particular, $\norm{A_n}=1$ and $\tr A_n=0$.
The equal magnitudes of its off-diagonal entries will
give both the diagonal commutator lower bound and the
optimal paving norms.

\subsection{A diagonal factor forces dimension dependence}

If $D=\diag(z_1,\ldots,z_n)$ and $A_n=[D,C]$, then
\[
(A_n)_{ij}=(z_i-z_j)c_{ij}.
\]
Thus nearby diagonal entries of $D$ force large entries
of $C$. The following estimate quantifies the accumulation
of such pairs when the points $z_i$ lie in a bounded square.

\begin{lemma}\label{lem:energy}
Let $n=4^k$, $k\ge1$, and let $z_1,\ldots,z_n$ be
distinct points in $Q=[-1,1]+\ii[-1,1]$. Then
\begin{equation}\label{eq:energy}
\sum_{i\ne j}\frac1{|z_i-z_j|^2}
\ge\frac{(3k-1)n^2+n}{32},
\end{equation}
where the sum is over ordered pairs.
\end{lemma}

\begin{proof}
For $\ell=0,\ldots,k-1$, subdivide $Q$ into $4^\ell$
squares of side length $2^{1-\ell}$, assigning boundary
points consistently to obtain nested partitions.
Let $m_{\ell,\alpha}$ be the number of points in the
$\alpha$th square. The number $N_\ell$ of ordered
distinct pairs lying in the same square satisfies
\begin{equation}\label{eq:count}
N_\ell
=\sum_\alpha m_{\ell,\alpha}^2-n
\ge\frac{n^2}{4^\ell}-n,
\end{equation}
by Cauchy--Schwarz.

Fix an ordered pair at distance $d>0$, and let
$J\subseteq\{0,\ldots,k-1\}$ be the set of levels at
which its two points lie in the same square.
For every $\ell\in J$, the diameter of that square gives
$d^2\le8\,4^{-\ell}$. Hence
\[
\sum_{\ell\in J}4^\ell
\le
\sum_{\substack{\ell\ge0\\4^\ell\le8/d^2}}4^\ell
\le\frac{32}{3d^2}.
\]
Summing over all ordered distinct pairs and applying
\eqref{eq:count}, we obtain
\begin{align*}
\sum_{i\ne j}|z_i-z_j|^{-2}
&\ge\frac3{32}\sum_{\ell=0}^{k-1}4^\ell N_\ell\\
&\ge\frac3{32}\sum_{\ell=0}^{k-1}(n^2-n4^\ell)\\
&=\frac3{32}
  \left(kn^2-\frac{n(n-1)}3\right)\\
&=\frac{(3k-1)n^2+n}{32}.
\end{align*}
\end{proof}

Now let $n=4^k$, $k\ge1$, and suppose
$A_n=[D,C]$, where $D=\diag(z_1,\ldots,z_n)$ has
entries in $Q$. Since every off-diagonal entry of $A_n$
is nonzero, the points $z_i$ must be distinct, and
\[
c_{ij}=\frac{(A_n)_{ij}}{z_i-z_j}\qquad(i\ne j).
\]

For any $T=(t_{ij})\in M_d(\C)$, the standard
orthonormal basis $e_1,\ldots,e_d$ gives
\begin{equation}\label{eq:entrynorm}
\sum_{i,j=1}^d |t_{ij}|^2
=\sum_{j=1}^d \norm{Te_j}^2
\le d\norm{T}^2.
\end{equation}
Applying this inequality to $C$, and using
\eqref{eq:flat} and Lemma~\ref{lem:energy}, we obtain
\begin{align*}
n\norm{C}^2
&\ge\sum_{i,j}|c_{ij}|^2\\
&\ge\frac1{n-1}\sum_{i\ne j}|z_i-z_j|^{-2}\\
&\ge\frac{(3k-1)n^2+n}{32(n-1)}.
\end{align*}
Since this estimate holds for every admissible $D$ and $C$,
\[
\lambda(A_n)\ge
\sqrt{\frac{(3k-1)n+1}{32(n-1)}}.
\]
Furthermore,
\[
\frac{(3k-1)n+1}{32(n-1)}-\frac{k}{16}
=\frac{(k-1)n+2k+1}{32(n-1)}>0.
\]
This proves part~(iii), namely \eqref{eq:mainlower}.

The same estimate forces dimension dependence for the
product of the two operator norms, without any normalization
on the diagonal factor. Indeed, suppose $A_n=[B,C]$
with $B$ diagonal. Since $A_n\ne0$, we have $\norm{B}>0$,
and
\[
\widetilde B=\frac{B}{\norm{B}},
\qquad
\widetilde C=\norm{B}\,C
\]
satisfy $A_n=[\widetilde B,\widetilde C]$.
The entries of $\widetilde B$ lie in the unit disk,
hence in $Q$, so
\[
\norm{B}\norm{C}
=\norm{\widetilde C}
\ge\lambda(A_n)
\ge\frac{\sqrt{k}}4.
\]
If instead $C$ is diagonal, apply the same argument to
$A_n=[C,-B]$. Thus the lower bound holds whenever
either factor is diagonal in the prescribed basis,
even when that factor is chosen after $A_n$.

Since $A_n$ has zero diagonal and norm one, we also obtain
\begin{equation}\label{eq:lambda}
\lambda(4^k)\ge\lambda(A_{4^k})
\ge
\sqrt{\frac{(3k-1)4^k+1}{32(4^k-1)}}
\ge\frac{\sqrt{k}}4.
\end{equation}
Consequently, $\sup_{m\ge2}\lambda(m)=\infty$.
Along $n=4^k$, this gives a lower bound of order
$\sqrt{\log n}$; no matching upper bound for
$\lambda(A_n)$ is asserted.

\subsection{Optimal pavings}

We now prove part~(i). For every nonempty coordinate
set $S\subseteq\{1,\ldots,n\}$, equations~\eqref{eq:entrynorm}
and~\eqref{eq:flat} give
\[
|S|\norm{(A_n)_S}^2
\ge\sum_{i,j\in S}|(A_n)_{ij}|^2
=\frac{|S|(|S|-1)}{n-1}.
\]
Dividing by $|S|$ gives
\begin{equation}\label{eq:compressionlower}
\norm{(A_n)_S}^2\ge\frac{|S|-1}{n-1}.
\end{equation}
Every partition into at most $r$ blocks contains a block
of size at least $\lceil n/r\rceil$. Therefore
\begin{equation}\label{eq:anyr}
p_r(A_n)\ge
\sqrt{\frac{\lceil n/r\rceil-1}{n-1}}
\qquad(1\le r\le n).
\end{equation}

For dyadic $r=2^\ell\le n$, the recursion attains this
lower bound. Taking the two diagonal blocks in
\eqref{eq:recursion} repeatedly, $\ell$ times, partitions
the coordinates into $r$ blocks of size $m=n/r$.
On each block, the corresponding principal submatrix
of $S_n$ is $S_m$ or $-S_m$. By
Lemma~\ref{lem:construction}, its compression in $A_n$
has norm
\[
\frac{\norm{S_m}}{\sqrt{n-1}}
=\sqrt{\frac{m-1}{n-1}}.
\]
This also holds for $m=1$, when the compression is zero.
Together with \eqref{eq:anyr}, this proves
\eqref{eq:optimalpaving}.

To obtain a block count independent of $n$, fix
$0<\varepsilon<1$ and let $R$ be the least power of two
with $R\ge\varepsilon^{-2}$. Then
$R<2\varepsilon^{-2}$.
If $n<R$, use the singleton partition. If $n\ge R$,
use the recursive $R$-block partition, for which
\[
p_R(A_n)^2
=\frac{n/R-1}{n-1}
\le\frac1R
\le\varepsilon^2.
\]
Thus every matrix in the family admits an explicit
$\varepsilon$-paving with at most $R<2\varepsilon^{-2}$
blocks.

The order $\varepsilon^{-2}$ is optimal for a block
count uniform over this family. Indeed,
\eqref{eq:compressionlower} implies that each block of
an $\varepsilon$-paving has size at most
$1+\varepsilon^2(n-1)$. Hence any such paving into
$r$ blocks satisfies
\[
r\ge\frac{n}{1+\varepsilon^2(n-1)}.
\]
The right-hand side tends to $\varepsilon^{-2}$
as $n\to\infty$ through powers of two.

\subsection{Optimal commutators with normal factors}

Finally, we prove part~(ii). Since $A_n$ is a trace-zero
Hermitian unitary, its $+1$ and $-1$ eigenspaces both
have dimension $n/2$. Thus there is a unitary $U$ such that
\[
U^*A_nU=
Z=\begin{pmatrix}I&0\\0&-I\end{pmatrix}.
\]
Here and below, the blocks have size $n/2$. Set
\[
X=\begin{pmatrix}0&I\\I&0\end{pmatrix},
\qquad
Y=\frac12\begin{pmatrix}0&-I\\I&0\end{pmatrix}.
\]
Direct multiplication gives $[X,Y]=Z$.
Moreover, $X$ is a selfadjoint unitary and $Y$ is
skew-adjoint, so both are normal and
\[
\norm{X}=1,\qquad \norm{Y}=\frac12.
\]
Consequently,
\[
A_n=[UXU^*,UYU^*]
\]
is a commutator of two normal matrices with norm
product $1/2$.

Conversely, every representation $A_n=[B,C]$ satisfies
\[
1=\norm{A_n}
=\norm{[B,C]}
\le2\norm{B}\norm{C}.
\]
Hence the minimum norm product is exactly $1/2$,
both for unrestricted factors and for two normal factors.
This completes the proof of Theorem~\ref{thm:main2}.

The example therefore isolates the effect of the prescribed
basis: its paving bounds and normal commutator cost are
uniform in the dimension, while requiring either factor
to be diagonal in that basis forces the commutator cost
to grow without bound.

\section{Formal verification of commutator bounds and the Kadison--Singer theorem}
\label{sec:formal}

We formalize the dimension-independent commutator bound of
Theorem~\ref{thm:main} and the Kadison--Singer theorem in
Lean~4~\cite{moura2021lean}, using Mathlib~\cite{mathlib}.
The former establishes a single commutator bound valid in
every matrix dimension; the latter establishes the uniqueness
of pure state extensions from the bounded diagonal algebra
to $B(\ell^2(\mathbb N;\mathbb C))$.
The development thus covers both uniform estimates for
finite-dimensional matrices and the infinite-dimensional
state-extension formulation of the Kadison--Singer problem.

The formalization also includes the quantitative theorem
of Marcus, Spielman, and Srivastava, the subpolynomial
commutator bounds of Johnson, Ozawa, and Schechtman, and
the explicit separation estimates of
Theorem~\ref{thm:main2}. These results make the scope of
the uniform bound precise: arbitrary commutator factors
admit a dimension-independent bound, whereas requiring a
factor to be diagonal in the prescribed basis leads to
the unbounded quantity $\lambda(m)$.
The formal statements explicitly record the uniform
quantifiers and the restrictions on the factors that
distinguish these conclusions.

All proofs are checked by Lean's kernel and depend only
on the standard axioms \texttt{propext},
\texttt{Classical.choice}, and \texttt{Quot.sound},
with no additional axioms or \texttt{sorryAx} dependencies.
The formalization is available online.\footnote{
\url{https://github.com/WuProver/lean-commutator}}
This section presents the principal formalized results
and their corresponding Lean declarations.


\subsection{Commutator bounds with subpolynomial dimension dependence}

Johnson, Ozawa, and Schechtman~\cite{johnson2013quantitative}
proved that for every $\varepsilon>0$, there exists a constant
$K_\varepsilon>0$, independent of the dimension $n$ and the
matrix $A$, such that every trace-zero $A\in M_n(\C)$ admits
a representation $A=\comm BC$ satisfying
\[
  \norm B\norm C\le K_\varepsilon n^\varepsilon\norm A.
\]
The first Lean declaration formalizes this estimate.
The second records the stronger conclusion that the first
factor $B$ can be chosen normal:

\begin{leancode}
theorem main_commutator_theorem (ε : ℝ) (hε : 0 < ε) :
    ∃ (Kε : ℝ), 0 < Kε ∧
    ∀ (n : ℕ) (A : Matrix (Fin n) (Fin n) ℂ),
      A.trace = 0 →
      ∃ (B C : Matrix (Fin n) (Fin n) ℂ),
        A = ⁅B, C⁆ₘ ∧
        ‖B‖ * ‖C‖ ≤ Kε * (n : ℝ) ^ ε * ‖A‖ := by ...

theorem main_commutator_theorem_normal (ε : ℝ) (hε : 0 < ε) :
    ∃ (Kε : ℝ), 0 < Kε ∧
    ∀ (n : ℕ) (A : Matrix (Fin n) (Fin n) ℂ),
      A.trace = 0 →
      ∃ (B C : Matrix (Fin n) (Fin n) ℂ),
        A = ⁅B, C⁆ₘ ∧
        B * B.conjTranspose = B.conjTranspose * B ∧
        ‖B‖ * ‖C‖ ≤ Kε * (n : ℝ) ^ ε * ‖A‖ := by ...
\end{leancode}

\subsection{The dimension-independent commutator bound}

Theorem~\ref{thm:main} establishes an absolute constant $K>0$,
independent of both the dimension $n$ and the input matrix $A$,
such that every trace-zero $A\in M_n(\C)$ admits a representation
$A=\comm BC$ with
\[
  \norm B\norm C\le K\norm A.
\]
The factors $B$ and $C$ may depend on $A$, with no normality
requirement. The Lean statement makes the uniformity explicit:
the constant $K$ is chosen before quantifying over $n$ and $A$.

\begin{leancode}
theorem uniformCommutatorBound :
    ∃ K : ℝ, 0 < K ∧ ∀ (n : ℕ) (A : Matrix (Fin n) (Fin n) ℂ),
      Matrix.trace A = 0 →
        ∃ B C : Matrix (Fin n) (Fin n) ℂ, A = B * C - C * B ∧ ‖B‖ * ‖C‖ ≤ K * ‖A‖ := by ...

\end{leancode}

\subsection{Diagonal commutators and the obstruction to a uniform bound }



Theorem~\ref{thm:main2} shows by an explicit construction that
the quantity $\lambda(m)$ of Johnson, Ozawa, and
Schechtman~\cite{johnson2013quantitative} cannot be bounded
independently of the matrix dimension. The construction uses
skew-Hadamard matrices to produce zero-diagonal Hermitian
unitaries $A_n$ satisfying
\[
  \lambda(4^k)\ge\lambda(A_{4^k})\ge\frac{\sqrt{k}}4
  \qquad(k\ge1).
\]
Hence $\sup_m\lambda(m)=\infty$. Crucially, this lower bound
holds even when the diagonal factor is chosen separately
for each input matrix.

The same family admits $\varepsilon$-pavings with fewer than
$2\varepsilon^{-2}$ blocks, uniformly in the dimension, and
its optimal commutator cost with both factors normal is
exactly $1/2$. Thus, these matrices have uniform paving and
normal commutator bounds, while their commutator cost becomes
unbounded when one factor must be diagonal in the prescribed
basis.

The example therefore rules out the strategy of proving
paving through a dimension-independent bound for $\lambda(m)$.
The conditional implication in Claim~3
of~\cite{johnson2013quantitative} remains valid, but its
uniform boundedness hypothesis cannot hold.

First, we formalize the zero-diagonal Hermitian unitaries $A_n$
for $n=2^m$ with $m\ge1$. We use a recursive index type
\texttt{Cube m} of cardinality $2^m$:
\begin{leancode}
@[reducible] def Cube : ℕ → Type
  | 0 => Fin 1
  | m + 1 => Cube m ⊕ Cube m

noncomputable def skewMatrix :
    (m : ℕ) → Matrix (Cube m) (Cube m) ℂ
  | 0 => 0
  | m + 1 => fromBlocks
      (skewMatrix m) (skewMatrix m + 1)
      (skewMatrix m - 1) (-skewMatrix m)

noncomputable def pavingMatrix (m : ℕ) :
    Matrix (Cube m) (Cube m) ℂ :=
  (Complex.I / (Real.sqrt ((2 : ℝ) ^ m - 1) : ℂ)) • skewMatrix m
\end{leancode}
Formalizing $S_{2^m}$ as $\texttt{skewMatrix m}$, the recursion gives
$(S_{2^m})_{ii}=0$, $S_{2^m}^*=-S_{2^m}$, and $S_{2^m}^2=-(2^m-1)I$.
Thus \texttt{pavingMatrix m} represents
$A_{2^m}=\mathrm{i}S_{2^m}/\sqrt{2^m-1}$, which has zero diagonal
and satisfies $A_{2^m}^*=A_{2^m}$ and $A_{2^m}^2=I$
for $m\ge1$, as required.

We use \texttt{pavingMinimum A r} to represent $p_r(A)$
defined in~\eqref{eq:PrA}:
\begin{leancode}
def colorClass (c : ι → κ) (a : κ) : Finset ι :=
  Finset.univ.filter (fun i ↦ c i = a)

def compression (s : Finset ι) (A : Matrix ι ι ℂ) : Matrix ι ι ℂ :=
  coordProjection s * A * coordProjection s

def pavingNorm (A : Matrix ι ι ℂ) (c : ι → κ) : ℝ :=
  Finset.univ.sup' Finset.univ_nonempty (fun a ↦ ‖compression (colorClass c a) A‖)

def pavingMinimum (A : Matrix ι ι ℂ) (r : ℕ) [NeZero r] : ℝ :=
  Finset.univ.inf' Finset.univ_nonempty
    (fun c : ι → Fin r ↦ pavingNorm A c)
\end{leancode}

We also use \texttt{lambdaA} to represent the $\lambda(A)$ defined in~\eqref{eq:lambda_A}:
\begin{leancode}
def diagonal [Zero α] (d : n → α) : Matrix n n α :=
  of fun i j => if i = j then d i else 0
  
def lambdaCosts {n : ℕ} (A : Matrix (Fin n) (Fin n) ℂ) : Set ℝ :=
  {t | ∃ z : Fin n → ℂ, ∃ C : Matrix (Fin n) (Fin n) ℂ,
    (∀ i, |(z i).re| ≤ 1 ∧ |(z i).im| ≤ 1) ∧
    A = Matrix.diagonal z * C - C * Matrix.diagonal z ∧ t = ‖C‖}

def lambdaA {n : ℕ} (A : Matrix (Fin n) (Fin n) ℂ) : ℝ := sInf (lambdaCosts A)
\end{leancode}

Then we formalize Theorem~\ref{thm:main2} as follows:

\begin{leancode}
def separationLowerBound (k : ℕ) : ℝ :=
  Real.sqrt ((((3 : ℝ) * k - 1) * (4 : ℝ) ^ k + 1) /
    (32 * ((4 : ℝ) ^ k - 1)))

theorem paving_commutator_separation :
    (∀ m : ℕ, 0 < m →
      Fintype.card (Cube m) = 2 ^ m ∧
      (∀ i, pavingMatrix m i i = 0) ∧ (pavingMatrix m).IsHermitian ∧
      pavingMatrix m * pavingMatrix m = 1 ∧ ‖pavingMatrix m‖ = 1) ∧
    (∀ m l : ℕ, 0 < m → l ≤ m →
      pavingMinimum (pavingMatrix m) (2 ^ l) =
        Real.sqrt ((((2 : ℝ) ^ m / (2 : ℝ) ^ l) - 1) / ((2 : ℝ) ^ m - 1))) ∧
    -- A single block budget, strictly below 2 / ε², works in every dimension.
    (∀ ε : ℝ, 0 < ε → ε < 1 →
      ∃ r : ℕ, 0 < r ∧ (r : ℝ) < 2 / ε ^ 2 ∧
        ∀ m : ℕ, 0 < m → ∃ c : Cube m → Fin r,
          ∀ a : Fin r, ‖compression (colorClass c a) (pavingMatrix m)‖ ≤ ε) ∧
    -- The optimal cost with both commutator factors normal is 1/2.
    (∀ m : ℕ, 0 < m → twoNormalCost (pavingMatrix m) = 1 / 2) ∧
    -- The paper's λ(A), on the coordinate-reindexed family, is at least √k / 4.
    (∀ k : ℕ, 1 ≤ k →
      separationLowerBound k ≤ lambdaA (finFamily k) ∧
      Real.sqrt (k : ℝ) / 4 ≤ separationLowerBound k) ∧
    -- The same precise lower bound holds for diagonal entries in the prescribed square.
    (∀ k : ℕ, 1 ≤ k →
      ∀ z : Cube (2 * k) → ℂ, ∀ C : Matrix (Cube (2 * k)) (Cube (2 * k)) ℂ,
        (∀ i, |(z i).re| ≤ 1 ∧ |(z i).im| ≤ 1) →
        pavingMatrix (2 * k) = Matrix.diagonal z * C - C * Matrix.diagonal z →
        separationLowerBound k ≤ ‖C‖) := by ...

theorem finFamily_lambdaA_lower_bound (k : ℕ) (hk : 1 ≤ k) :
    Real.sqrt ((((3 : ℝ) * k - 1) * (4 : ℝ) ^ k + 1) /
      (32 * ((4 : ℝ) ^ k - 1))) ≤ lambdaA (finFamily k) := by ...
\end{leancode}

\subsection{The Kadison--Singer theorem}

Our formalization of the Kadison--Singer theorem builds on the
main result of Marcus, Spielman, and
Srivastava~\cite[Theorem~1.4]{marcus2015interlacing},
stated here as Theorem~\ref{thm:MSS}.
Its Lean formulation is as follows:
\begin{leancode}
def vecMulVec [Mul α] (w : m → α) (v : n → α) : Matrix m n α :=
  of fun x y => w x * v y

namespace MSSSelection in
noncomputable def energy (v : ι → ℂ) : ℝ := ∑ i, ‖v i‖ ^ 2

theorem finite_mss {κ : Type*} [Fintype κ]
    (v : κ → Ω → ι → ℂ) (p : κ → Ω → ℝ)
    (hp : ∀ i ω, 0 ≤ p i ω) (hsum : ∀ i, ∑ ω, p i ω = 1)
    (hTotal : (∑ i, ∑ ω, (p i ω : ℂ) •
      Matrix.vecMulVec (v i ω) (star (v i ω))) = 1)
    (ε : ℝ) (hε : 0 < ε)
    (henergy : ∀ i, ∑ ω, p i ω * MSSSelection.energy (v i ω) ≤ ε) :
    ∃ q : κ → Ω,
      ‖∑ i, Matrix.vecMulVec (v i (q i)) (star (v i (q i)))‖ ≤ (1 + Real.sqrt ε)^2 := by ...
\end{leancode}




Let $H=\ell^2(\mathbb N;\mathbb C)$, and let
$\mathcal D\subset B(H)$ denote the algebra of bounded operators
that are diagonal with respect to the standard orthonormal basis.
Kadison and Singer~\cite{kadison1959extensions} asked whether
every pure state on $\mathcal D$ admits a unique pure state
extension to $B(H)$. This problem was reformulated in terms of
paving by Anderson~\cite{anderson1979extensions} and resolved
affirmatively by Marcus, Spielman, and
Srivastava~\cite{marcus2015interlacing}.

\begin{theorem}\label{thm:kadison-singer}
For every pure state $\varphi$ on $\mathcal D$, there exists
a unique pure state $\psi$ on $B(H)$ such that
\[
  \psi(d)=\varphi(d)\qquad\text{for every }d\in\mathcal D.
\]
\end{theorem}

We formalize
Theorem~\ref{thm:kadison-singer} in Lean as follows and prove it using Theorem~\ref{thm:MSS}.
\begin{leancode}
def State := {φ : A →ₚ[ℂ] ℂ // φ 1 = 1}

def lp (E : α → Type*) [∀ i, NormedAddCommGroup (E i)] (p : ℝ≥0∞) : AddSubgroup (PreLp E) where
  carrier := { f | Memℓp f p }
  zero_mem' := zero_memℓp
  add_mem' := Memℓp.add
  neg_mem' := Memℓp.neg

abbrev Operator := Hilbert →L[ℂ] Hilbert
abbrev Diagonal : Type := lp (fun _ : ℕ ↦ ℂ) ∞

def IsPure (φ : State A) : Prop :=
  ∀ (ψ χ : State A) (t : ℝ), 0 < t → t < 1 →
    (∀ a, φ.val a = t • ψ.val a + (1 - t) • χ.val a) → ψ = φ

def diagonalRepresentation : Diagonal →⋆ₐ[ℂ] Operator where
  toFun := diagonalOperator
  map_zero' := by ext x i; change (0 : ℂ) * x i = 0; simp
  map_one' := by ext x i; simp
  map_add' a b := by
    ext x i
    change (a i + b i) * x i = a i * x i + b i * x i
    ring
  map_mul' a b := by ext x i; simp [mul_assoc]
  commutes' c := by ext x i; simp [Algebra.algebraMap_eq_smul_one]
  map_star' := diagonalOperator_star

theorem kadison_singer (φ : State Diagonal) (hφ : φ.IsPure) :
    ∃! ψ : State Operator, ψ.IsPure ∧ ∀ d : Diagonal,
      ψ.val (diagonalRepresentation d) = φ.val d := by ...
\end{leancode}

\section*{Declaration of Generative AI}
The authors acknowledge the use of GPT-6 Astra and GPT-5.6 Sol during the work, which helps in exploring
proof strategies and finishing the formalization in Lean~4. The authors independently checked and verified all proofs and mathematical details, confirmed that the informal statement of the main theorem is faithfully aligned with its formal counterpart, and refactored parts of the Lean code where necessary. The authors assume full responsibility for the mathematical correctness of the final results.

\section*{Acknowledgements}
The authors  are supported by the National Key R\&D Program of China 2023YFA1009401 and the Strategic Priority Research Program of Chinese Academy of Sciences under Grant XDA0480501. 

\bibliographystyle{cas-model2-names}

\bibliography{cas-refs}
\end{document}